\documentclass[11pt,reqno]{amsart}

\usepackage{amsmath,amsthm,amssymb}
\usepackage{amssymb}

\usepackage{graphics,graphicx}
\usepackage{hyperref}
\usepackage[usenames, dvipsnames]{xcolor}

\definecolor{darkblue}{rgb}{0.0,0.0,0.3}
\hypersetup{colorlinks,breaklinks,
  linkcolor=darkblue,urlcolor=darkblue,
anchorcolor=darkblue,citecolor=darkblue}

\usepackage[square,sort,comma,numbers]{natbib}
\usepackage{fancyvrb}
\usepackage{enumitem}

\usepackage{booktabs}

\theoremstyle{plain}
\newtheorem{theorem}{Theorem}[section]
\newtheorem*{theorem*}{Theorem}
\newtheorem{lemma}[theorem]{Lemma}
\newtheorem{proposition}[theorem]{Proposition}
\newtheorem*{proposition*}{Proposition}

\newtheorem*{corollary*}{Corollary}
\newtheorem{claim}[theorem]{Claim}

\theoremstyle{definition}

\newtheorem{example}[theorem]{Example}

\numberwithin{equation}{section}

\newtheoremstyle{named}{}{}{\itshape}{}{\bfseries}{.}{.5em}{\thmnote{#3}}
\theoremstyle{named}

\title{Arithmetic Progressions with Restricted Digits}

\author[Aled Walker]{Aled Walker}
\address{Trinity College, Cambridge, CB2 1TQ, United Kingdom}
\email {aledwalker@gmail.com}
\author[Alexander Walker]{Alexander Walker}
\address{Department of Mathematics, Rutgers University, Hill Center - Busch Campus, 110 Frelinghuysen Road, Piscataway, NJ 08854-8019, USA}
\begin{document}

\begin{abstract}
For an integer $b \geqslant 2$ and a set $S\subset \{0,\cdots,b-1\}$, we define the \emph{Kempner set} $\mathcal{K}(S,b)$ to be the set of all non-negative integers whose base-$b$ digital expansions contain only digits from $S$. These well-studied sparse sets provide a rich setting for additive number theory, and in this paper we study various questions relating to the appearance of arithmetic progressions in these sets. In particular, for all $b$ we determine exactly 
the maximal length of an arithmetic progression that omits a base-$b$ digit.
\end{abstract}

\maketitle

\section{Introduction}

In 1914 Kempner \cite{Ke14} introduced a variant of the harmonic series which excluded from its sum all those positive integers that contain the digit $9$ in their base-$10$ expansions.  Unlike the familiar harmonic series, Kempner's modified series converges (the limit later shown to be $\approx 22.92$, see \cite{Ba79}).  A simple generalisation of Kempner's original argument shows that convergence occurs as long as any non-empty set of digits is excluded, and that this result holds in any base (see \cite{BS08}, for example).

Let us introduce some notation to describe these results in general. Fix an integer $b \geqslant 2$ and a subset of integers $S \subseteq [0,b-1]$. Here and throughout the paper, for two integers $x$ and $y$ we use $[x,y]$ to denote the set $\{n\in\mathbb{Z}:x \leqslant n\leqslant y\}$.  We then define the \emph{Kempner set} $\mathcal{K}(S,b)$ to be the set of non-negative integers that, when written in base $b$, contain only digits from $S$. Thus $\mathcal{K}([0,8],10)$ denotes the set originally studied by Kempner. We will assume throughout that $0 \in S$, to avoid the ambiguity of leading zeros, and require $S \neq [0,b-1]$, to preclude the trivial set $\mathcal{K}([0,b-1],b)$ (which is nothing more than $\mathbb{Z}_{\geqslant 0}$). These sets $S$ will be referred to as the \emph{permitted} sets $S$, and the related Kempner sets $K(S,b)$ as \emph{proper} Kempner sets. 

The arithmetic properties of proper Kempner sets have been the object of considerable study in recent years, beginning with the work of Erd\H{o}s, Mauduit, and S\'ark\"ozy, who studied the distribution of residues in $\mathcal{K}(S,b)$ moduli small numbers \cite{EMS98} and proved the existence of integers in $\mathcal{K}(S,b)$ with many small prime factors \cite{EMS99}.
Notable recent work includes Maynard's proof \cite{May16} that the sets $\mathcal{K}(S,b)$ contain infinitely many primes whenever $b-\vert S\vert$ is at most $b^{23/80}$, provided $b$ is sufficiently large. 

In this paper we consider the additive structure of proper Kempner sets. In particular, we consider the following extremal question: \emph{what is the length of the longest arithmetic progression in a proper Kempner set with a fixed given base?} Our methods will be combinatorial, rather than analytic (as in Maynard's work, \cite{May16}). 

A well known conjecture of Erd\H{o}s-Tur\'{a}n (first given in \cite{ErTu36}) states that any set of positive integers with a divergent harmonic sum contains arithmetic progressions of arbitrary (finite) length. Since proper Kempner sets have convergent harmonic sums, this might suggest that the lengths of arithmetic progressions in a given proper Kempner set are uniformly bounded.

This is indeed the case.  Let us say that a set $T\subset \mathbb{Z}$ is $k$-free if $T$ contains no arithmetic progression of length $k$. By a simple argument, given in Proposition \ref{prop:ell_finite}, one may show that the proper Kempner set $\mathcal{K}(S,b)$ is $(b^2-b+1)$-free for any $b \geqslant 2$.

The main purpose of this article is to understand how close this trivial upper bound is to the truth. \\

In our main theorem, we improve this bound for all $b> 2$, obtaining a tight result that expresses the length of the longest arithmetic progression in $\mathcal{K}(S,b)$ in terms of the prime factorisation of $b$. To state this theorem, we need to introduce some arithmetic functions. If $n$ and $b$ are natural numbers, let $\rho(n)$ denote the square-free radical of $n$ (ie. the product of all distinct primes dividing $n$), and let $\beta(b)$ denote the largest integer less than $b$ such that $\rho(\beta(b))\vert b$.
For example, $\beta(10) = 8$, and $\beta(p^k) = p^{k-1}$ for any prime power $p^k$. In other words, $\beta(b)$ is the greatest integer less than $b$ that divides some power of $b$. Finally, let $\ell(b)$ be the length of the longest arithmetic progression contained in some proper Kempner set of base $b$. 

Our main theorem gives an exact evaluation of $\ell(b)$.

\begin{theorem}
\label{thm:main theorem} For all $b\geqslant 2$, one has $\ell(b) = (b-1)\beta(b)$.
\end{theorem}
\noindent For example, $\ell(10)=72$. One particular set that achieves this bound is Kempner's original set, $\mathcal{K}([0,8],10)$, which contains the $72$-term arithmetic progression $\{0,125,250,375,\cdots,8875\}$.\\

The arithmetic functions $\beta(b)$ and $\ell(b)$ are of independent interest, but do not appear to have been considered seriously before.\footnote{The sequence $\beta(b)$ is entry A079277 on the Online Encyclopedia of integer Sequences.} We establish average order results for $\beta(b)$ which show that, for most $b$, the trivial upper bound on $\ell(b)$ from Proposition \ref{prop:ell_finite} is asymptotically correct. 

\begin{theorem}
\label{thm: ell nearly always maximal epsilon free}
There is a set of integers $A\subset \mathbb{Z}$ with natural density $1$, i.e. with \[ \lim\limits_{N\rightarrow \infty} \frac{1}{N}\vert A\cap [1,N]\vert  = 1,\] such that $\ell(b)\sim b^2$ as $b\rightarrow \infty$ in $A$.
\end{theorem}
\noindent 
\textit{Notation:} For $x\in \mathbb{R}$, let $\{x\}$ denote the fractional part of $x$ and let $ \lfloor x \rfloor$ denote the greatest integer that is at most $x$. For a natural number $n$, we let $[n]$ denote the set of integers $\{1,\cdots,n\}$. As mentioned previously, for two integers $x$ and $y$ we use $[x,y]$ to denote the set $\{ n\in \mathbb{Z}: x\leqslant n\leqslant y\}$. We use the notation $\log_q p$ to denote the logarithm of $p$ to base $q$ (as opposed to any iterations of logarithms).\\

\section{Progressions of Maximal Length in Kempner Sets}

In this section we give our proof of Theorem \ref{thm:main theorem}, which is an exact evaluation of $\ell(b)$ and the main result of this paper.  This will be done in two parts: a constructive lower bound and a proof that this lower bound is sharp.  Before that, as promised, we give a simple proof that the function $\ell(b)$ is at least well-defined, i.e. that Kempner sets do not contain arbitrarily long arithmetic progressions.

\begin{proposition} \label{prop:ell_finite}
For all $b\geqslant 2$, we have $\ell(b) \leqslant (b-1)b$.
\end{proposition}

\begin{proof}  Suppose that $A \subset \mathcal{K}(S,b)$ is a finite arithmetic progression of $\vert A\vert$ terms with common difference $\Delta$. Choose $k \geqslant 0$ such that $b^k \leqslant \Delta < b^{k+1}$.  If $I$ denotes the shortest interval of integers containing $A$, then $\vert I \vert = (\vert A \vert-1)\Delta+1$, hence $\vert A\vert = 1+ (\vert I\vert -1)/\Delta$.

If $A$ excludes the digit $d$, the upper bound $\Delta < b^{k+1}$ confines $A$ within the interval $[0,db^{k+1} - 1]$ or within an interval of the form
\[[b^{k+2} m + (d+1) b^{k+1}, b^{k+2}(m+1)+ db^{k+1}-1],\]
for some $m\in \mathbb{Z}_{\geqslant 0}$. Thus $\vert I\vert \leqslant  b^{k+2}-b^{k+1}$, which yields
\[\vert A\vert \leqslant 1 + \frac{b^{k+2}-b^{k+1}-1}{\Delta} < 1 + \frac{b^{k+2} - b^{k+1}}{b^{k}} \leqslant b^2-b+1,\]
hence $\vert A\vert  \leqslant b^2-b$ as claimed.
\end{proof}

The bound in the previous proposition is simple and -- as a consequence -- occasionally weak.  In particular, it neglects the potentially compounding effects of digit exclusion at different orders of magnitude, and the arithmetic properties of orbits in the group $\mathbb{Z}/b\mathbb{Z}$. This structure can affect the bounds dramatically, as seen most clearly in the case when the base $b$ is prime.

\begin{proposition}
\label{prime prop}
Let $p$ be prime. Then $\ell(p) \leqslant p-1$. 
\end{proposition}
\begin{proof} Suppose that $\mathcal{K}(S,p)$ contains the progression $A = \{k+ j \Delta : j \in [p]\}$ with $\Delta\neq 0$.  By the pigeonhole principle, there exist distinct $i,j\in [p]$ with $k+j \Delta \equiv k+i \Delta \!\!\mod p$ for some $i\neq j$, hence $p \mid \Delta$ (since $p$ is prime).  By deleting the rightmost digits of the elements of $A$ we obtain a new progression in $\mathcal{K}(S,p)$ with common difference $\Delta/p$; in particular, the progression
$$\left\{\left\lfloor \frac{k}{p} \right\rfloor + j\frac{\Delta}{p}: j\in [p]\right\}.$$
The new common difference is strictly smaller, and we obtain a contradiction by infinite descent.
\end{proof}

With a little more bookkeeping this proof generalizes to prime powers, and implies that $\ell(p^k) \leqslant p^{k-1}(p^k-1)$. So certainly $\ell(b)$ is not asymptotic to $b^2$ as $b$ ranges over all integers; some restriction in Theorem \ref{thm: ell nearly always maximal epsilon free} is required.\\

We now begin the proof of Theorem \ref{thm:main theorem}. Searching for long progressions in $\mathcal{K}([0,8],10)$, one might happen across the example noted earlier, namely the first $71$ multiples of $125$, which -- together with $0$ -- form an arithmetic progression of length $72$, none of whose members contain the digit 9. This example succeeds due to properties of the prime factorisation of $1000/125$, in relation to the base $10$. These properties generalise, and one may use this to construct long digit-excluding arithmetic progressions in arbitrary bases.

\begin{proposition}
\label{prop:AP construction} For all $b \geqslant 2$, the Kempner set $\mathcal{K}([0,b-2],b)$ contains an arithmetic progression of length $(b-1)\beta(b)$. Hence $\ell(b) \geqslant (b-1)\beta(b)$.
\end{proposition}



\begin{proof}
Let $K\geqslant 1$ be the smallest natural number such that $\beta(b)\vert b^K$. We claim that all the members of the arithmetic progression $$ A = \frac{b^K}{\beta(b)}[0,(b-1)\beta(b)-1]$$ exclude the digit $b-1$ from their base-$b$ expansions. To see this, let $k$ satisfy $0\leqslant k\leqslant K-1$. 
Then $\gcd(b^{k+1}\beta(b),b^{K})\geqslant b^{k+1}>b^k\beta(b)$, which implies that $\gcd(b^{k+1},b^K/\beta(b))>b^k$ (by dividing through by $\beta(b)$).
In particular, for all integers $x$ and $y$, either
\begin{equation}
\label{separation condition}
\left\vert x \frac{b^K}{\beta(b)} - y b^{k+1}\right\vert>b^k \qquad \text{or} \qquad x  \frac{b^K}{\beta(b)} = y b^{k+1}.
\end{equation}

This observation implies that none of the $K$ rightmost digits of any integer of the form $xb^K/\beta(b)$ can be equal to $b-1$. Indeed, in base $b$, the $b^{k}$ digit of $x b^K/\beta(b)$ is the unique integer $d$ in the range $0\leqslant d\leqslant b-1$ such that
\[\left\{\frac{xb^K/\beta(b)}{b^{k+1}} \right\}\in \left[\frac{d}{b},\frac{d+1}{b}\right).\]
Yet~\eqref{separation condition} implies that $\{\frac{x b^K}{b^{k+1}\beta(b)}\} \in \{0\} \cup (\frac{1}{b}, \frac{b-1}{b})$ for each $0 \leqslant k \leqslant K-1$.  Since this is disjoint from $[\frac{b-1}{b},1)$, we conclude that none of the $K$ rightmost digits of any integer of the form $x b^K/\beta(b)$ can be equal to $b-1$.

We now fix $x \in [0,(b-1)\beta(b)-1]$ and consider the leftmost digits of $x b^K/\beta(b)$. Certainly $x b^K/\beta(b) < (b-1)b^K$.  From this upper bound we see that the $b^{K}$ digit of $x b^K/\beta(b)$ lies in $[0,b-2]$ and that the digits associated to larger powers of $b$ are all $0$.  Combining this with our previous observations, we conclude that $xb^K/\beta(b)$ omits the digit $(b-1)$ for all $x \in [0,(b-1) \beta(b)-1]$, so $A \subset \mathcal{K}([0,b-2],b)$ as claimed.  Since $\vert A\vert = (b-1)\beta(b)$, we have $\ell(b) \geqslant (b-1)\beta(b)$.
\end{proof}

We now proceed with the second half of our evaluation of $\ell(b)$, the verification that this lower bound is exact. This requires a more technical argument. 

\begin{proposition} \label{prop:upper_bound}
For all $b\geqslant 2$, we have $\ell(b) \leqslant (b-1)\beta(b)$.
\end{proposition}

\begin{proof}
Without loss of generality, let $S\subset [0,b-1]$ be any set of $b-1$ digits (containing $0$), and let $A = \{x+ j \Delta : j \in [0,\ell(b)-1]\}$ be an arithmetic progression in $\mathcal{K}(S,b)$ of maximal length, in which $\Delta > 0$ is taken minimally over all arithmetic progressions of length $\ell(b)$.

Let $\Delta = d_K b^K + \ldots + d_1 b + d_0$ denote the base $b$ expansion of $\Delta$, where $K$ is chosen such that $d_K \neq 0$.  For notational convenience, let $\Delta_k:= d_k b^k + \ldots +d_1b + d_0$ for each $k \geqslant 0$. (Note that $\Delta_k = \Delta$ for $k \geqslant K$.) We may assume without loss of generality that $d_0 \neq 0$, else by removing the rightmost digit from all elements of $A$ one constructs an arithmetic progression contained in $\mathcal{K}(S,b)$ of common difference $\Delta/b$, contradicting our minimality assumption on $\Delta$. (This is the same device as we used in the proof of Proposition \ref{prime prop}).

Our proof of Proposition~\ref{prop:upper_bound} rests on the following claim, whose peculiar statement arises naturally from an inductive argument. 
\begin{claim}
Consider the following statements:
\begin{itemize}
\item[\emph{C}$1$:] $\ell(b)\leqslant (b-1)\beta(b)$;
\item[\emph{C}$2$\emph{(}k\emph{)}:] there exist coprime integers $\lambda_k,\mu_k \in [1,b-1]$ satisfying\\ $\lambda_k \Delta_k = \mu_k b^{k+1}$. 
\end{itemize}
Then either \emph{C1} holds or \emph{C2(}k\emph{)} holds for all $k \geqslant 0$.
\end{claim}

This claim immediately settles the theorem, since the statement C2($k$) cannot possibly hold for all $k\geqslant 0$. Indeed, we have $\lambda_k \Delta_k < b \Delta$, while $\mu_k b^{k+1}$ grows in $k$ without bound.
\end{proof}

\begin{proof}[Proof of Claim]
We prove this claim by induction, showing that for every $k\geqslant 0$, either C1 holds or C2($k^\prime$) holds for all $k^\prime\leqslant k$. For the base case $k=0$, note that $\Delta_k = d_0$.  If $(d_0,b)=1$, then $d_0$ generates the additive group $\mathbb{Z}/b \mathbb{Z}$ and the elements $\{x + j \Delta : j \in [0,b-1]\}$ have $b$ distinct units digits.  Thus $\ell(b) \leqslant (b-1) \leqslant (b-1)\beta(b)$, so C1 holds. 

Otherwise, $(d_0,b)>1$, which implies that there exists $\lambda \in [1,b-1]$ for which $\lambda d_0 \equiv 0 \!\! \mod b$.  Thus $\lambda d_0 = \mu b$ for some $\mu$, and we may assume that $(\lambda,\mu)=1$ by dividing through by common factors.  This concludes the base case.

Proceeding to the inductive step, let $k\geqslant 1$ and assume that the inductive hypothesis C2($k^\prime$) holds for all smaller $k^\prime$. In particular, $\Delta_{k-1} = (\mu_{k-1}/\lambda_{k-1}) b^k$ for some coprime integers $\lambda_{k-1},\mu_{k-1}\in [1,b-1]$, and hence $\Delta_k = d_k b^k + (\mu_{k-1}/\lambda_{k-1}) b^k$.

Let $\lambda_k$ denote the order of $\Delta_k/b^{k+1}$ in the additive group $\mathbb{R}/\mathbb{Z}$, and let $\mu_k$ denote the integer $\lambda_k (\Delta_k/b^{k+1})$. We see that $(\lambda_k,\mu_k)=1$, as one could divide through by any common factors of $\lambda_k$ and $\mu_k$ to contradict the fact that $\lambda_k$ is the order of $\Delta_k/b^{k+1}$ in $\mathbb{R}/\mathbb{Z}$. Now, if $\lambda_k < b$, then $\mu_k < b$ as well, since $\Delta_k/b_{k+1}<1$ for any $k$.  In this case, $\lambda_k$ and $\mu_k$ satisfy the conditions listed in C2($k$). Therefore C2($k^\prime$) holds for all $k^\prime \leqslant k$.  

It remains to address the case $\lambda_k \geqslant b$. By usual facts about finite subgroups of $\mathbb{R}/\mathbb{Z}$, we note that the orbit of $\Delta_k/b^{k+1}$ in $\mathbb{R}/\mathbb{Z}$ is exactly the set of fractions with denominator dividing $\lambda_k$.  In particular, the set of values
\[T=\left\{\frac{x}{b^{k+1}} + \frac{\Delta_k j}{b^{k+1}} \!\! \mod 1 : j \in [0,\lambda_k -1]\right\}\]
are equally spaced, with gaps of size $1/\lambda_k$.  Since $\lambda_k \geqslant b$, for any integer $d \in [0,b-1]$ at least one member of $T$ lies in the half-open interval $[\frac{d}{b},\frac{d+1}{b})$. In other words, at least one member of the progression $x + \Delta_k[0,\lambda_k-1]$ has $b^{k}$ digit equal to $d$.

This information immediately implies that $x+\Delta[0,\lambda_k-1]$ is not contained in any proper Kempner set $\mathcal{K}(S,b)$, and hence $\ell(b) \leqslant \lambda_k -1$.  However, more can be said with a slight refinement to our analysis.  Equal spacing implies that at least $\lfloor \lambda_k/b \rfloor$ members of $T$ lie in the interval $[\frac{d}{b},\frac{d+1}{b})$. We are left with the stronger bound $\ell(b) \leqslant \lambda_k - \lfloor \lambda_k/b\rfloor$.

We now establish an upper bound on the function $\lambda_k - \lfloor \lambda_k/b\rfloor$, given the known constraints on $\lambda_k$.  For starters, the inductive hypothesis implies that $\lambda_{k-1} \mid \mu_{k-1} b^{k}$, hence $\lambda_{k-1} \mid b^{k}$ (since $\lambda_{k-1}$ and $\mu_{k-1}$ are coprime).  Since $\lambda_{k-1} < b$ and $\lambda_{k-1}$ divides a power of $b$, this implies that $\lambda_{k-1} \leqslant \beta(b)$.  Secondly, the inductive hypothesis allows us to write
\[\frac{\Delta_k}{b^{k+1}} = \frac{d_k \lambda_{k-1} + \mu_{k-1}}{b \lambda_{k-1}},\]
which implies that $b\lambda_{k-1} (\Delta_k/b^{k+1}) \equiv 0 \!\! \mod 1$.  This implies that $b \lambda_{k-1}$ is a multiple of the order of $(\Delta_k/b^{k+1}) \!\!\mod 1$, ie. $\lambda_k \mid b \lambda_{k-1}$.  We conclude that $\lambda_k \leqslant b \lambda_{k-1} \leqslant b\beta(b)$.

The function $\lambda \mapsto \lambda - \lfloor \lambda/b \rfloor$ is non-decreasing as $\lambda$ increases over integers, hence
\[\ell(b) \leqslant \lambda_k - \left\lfloor \frac{\lambda_k}{b}\right\rfloor \leqslant b \beta(b) - \left\lfloor \frac{b \beta(b)}{b}\right\rfloor =b \beta(b) - \beta(b)= (b-1)\beta(b),\]
which implies that C1 holds. This completes the inductive step, and so completes the proof of Theorem \ref{thm:main theorem}.
\end{proof}

\section{Asymptotic Analysis}

In this section we analyse the function $\beta(b)$, with the ultimate goal of proving Theorem \ref{thm: ell nearly always maximal epsilon free}. We begin with the following simple observation.

\begin{proposition} \label{prop:betabounds1}
We have
\[\liminf_{n \to \infty} \frac{\beta(n)}{n} =0 \quad \text{and} \quad \limsup_{n \to \infty} \frac{\beta(n)}{n} =1.\]
\end{proposition}

\begin{proof}
The first claim follows from the observation that $\beta(p) = 1$ for all primes $p$.  For the second, we note that $\beta(2^k+2)=2^k$ for all $k>1$.
\end{proof}

It is clear from this proposition that the behaviour of $\beta(n)$ is erratic as $n$ varies. However, its calculation may be understood as a certain integer programming problem, as illustrated by the following example. 

\begin{example} \label{example:beta24}
In this example, we calculate $\beta(24)$ using techniques from mixed integer programming.  We may write $\beta(24)=2^a\cdot 3^b$, with $a,b \in \mathbb{N}$.  It follows that $a \log 2 + b \log 3 < \log 24$, and $(a,b)$ may be visualized as a lattice point in the following figure (Figure \ref{fig1}). The equation of the line is $f(x) = \log_3 24-x\log_3 2 $. 

\begin{figure}[h]
\centering
\includegraphics[width=0.8\textwidth]{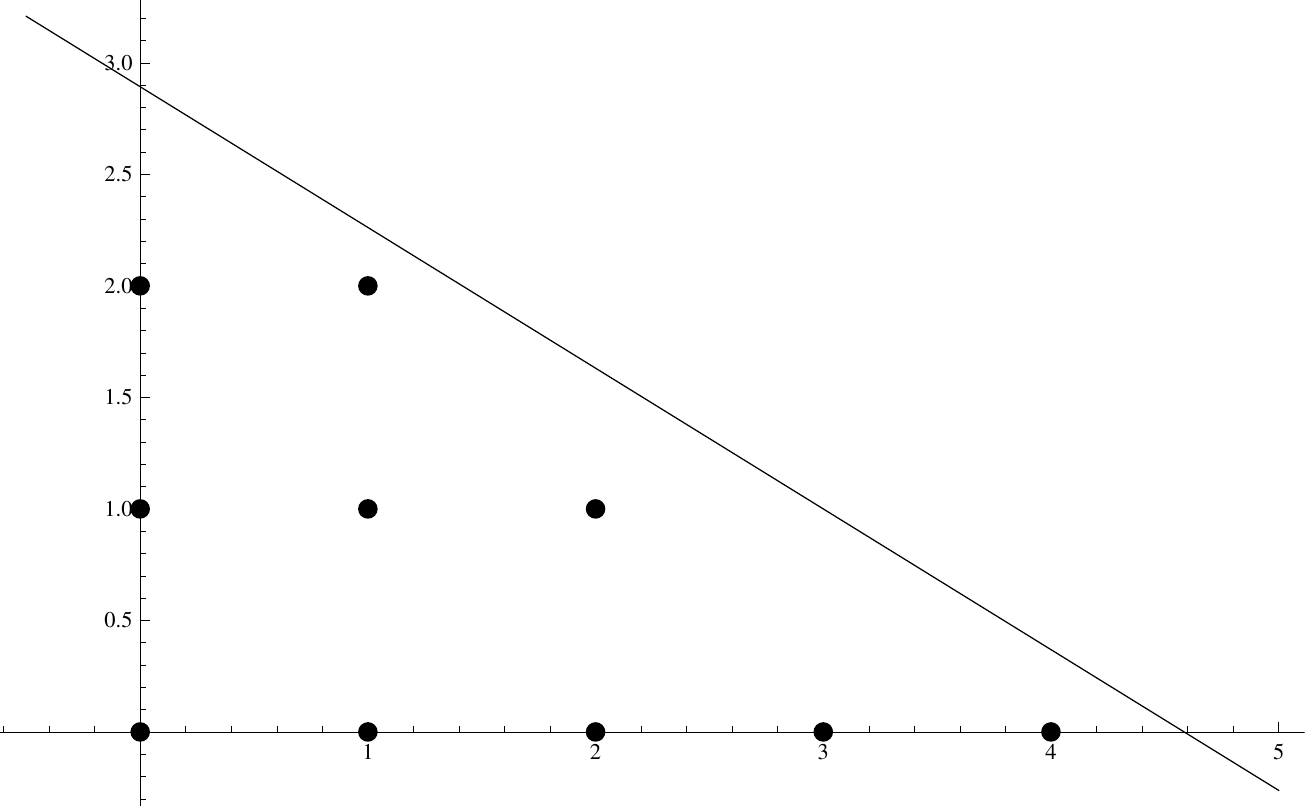}
\caption{Lattice points $(a,b)$ corresponding to $\beta(24)=2^a\cdot3^b$.}
\label{fig1}
\end{figure}
Let us restrict our attention to the set $S$ of lattice points of the form $(a,b)$, in which $b$ is taken maximally for fixed $a$.  If $(a,b) \in S$, the vertical distance from $(a,b)$ to the diagonal in Figure \ref{fig1} is then given by $\{\log_3 24 - a\log_3 2\}$. We also note that $a\log 2 + b \log 3$ is maximized among the lattice points below the line when $(a,b) \in S$ and $\{\log_3 24 - a \log_3 2\}$ is minimized (as a function of $a$). In our example, minimization occurs at $(a,b)=(1,2)$, and so we obtain $\beta(24)=2^1 \cdot 3^2=18$.
\end{example}

The technique of Example \ref{example:beta24} generalizes easily: if $n$ has $k$ prime divisors $p_1,\ldots,p_k$, we may associate to $n$ a set of lattice points in $\mathbb{Z}^k$, namely \[ \{ (a_1,\cdots,a_k) \in \mathbb{Z}_{\geqslant 0}^k : a_1\log p_1 + \cdots + a_k \log p_k < \log n\}.\] The lattice point $(a_1,\ldots, a_k)$ that minimizes distance to the the hyperplane
\[x_1 \log p_1 + \ldots x_k \log p_k = \log n\]
determines $\beta(n)$ by the formula $\beta(n)= \prod_{i=1}^k p_i^{a_i}$.  \\

Combining this idea with well-known equidistribution results gives the following.

\begin{lemma} \label{thm:beta_asymp_along_AP}
We have $\beta(n) \sim n$ as $n \to \infty$ within $N\mathbb{Z}$ if and only if $N$ is not a prime power.
\end{lemma}


\begin{proof}
If $N=p^k$ is a prime power, then $N\mathbb{Z}$ contains the subsequence $\{p^{kn}\}_{n \geqslant 1}$.  Since $\beta(p^{kn})=p^{kn-1}$, we cannot have $\beta(n) \sim n$ within $N\mathbb{Z}$.

Otherwise, let $p$ and $q$ be distinct primes dividing $N$, and fix a positive constant $\varepsilon$. As $\log_q p$ is irrational, the sequence $(\{u_n\})_{n=0}^\infty$ given by $u_n := n\cdot\log_q p$ is equidistributed mod $1$ (by the Equidistribution Theorem: see Proposition 21.1 of \cite{IwKo04}, say). In particular, there exists a positive parameter $L_\varepsilon$ such that $l>L_\varepsilon$ implies that the sequence $(\{ u_n\})_{n=0}^l$ contains at least one element in each interval mod $1$ of length $\varepsilon$.

Now let $m$ be a natural number and let $l = \lfloor \log_p m \rfloor$. From the above remarks, there exists a positive parameter $M_\varepsilon$ such that, for each $m>M_\varepsilon$, the shifted sequence $(\{\log_q m - u_n \})_{n=0}^l$ contains some element in the interval $(0,\varepsilon)$. In other words there exists $n_0$ at most $l$ (but dependent on $l$) such that
\[0 < \{\log_q m - n_0 \cdot \log_q p\}<\varepsilon.\] Also note that $\log_q m - n_0 \cdot \log_q p $ is positive. 

Now, assume $pq \mid m$ and consider $(a,b):=(n_0,\lfloor \log_q m - n_0 \cdot \log_q p \rfloor )$. We have $\beta(m) \geqslant p^a \cdot q^b$ by construction. So
\[\beta(m)\geqslant p^a \cdot q^b = q^{\log_q m - \{\log_q m - n_0 \cdot \log_q p \}} > q^{\log_q m - \varepsilon}=m \cdot q^{-\varepsilon}.\]
Thus $q^{-\varepsilon} < \beta(m)/m < 1$, for all $m$ satisfying $m>M_\varepsilon$ and $pq \mid m$.  Since $\varepsilon$ was arbitrary, and $q$ fixed, it follows that $\beta(m) \sim m$ within $pq\mathbb{Z}$, and hence within $N\mathbb{Z}$. 
\end{proof}

By considering $N=6$, for example, we obtain a set of density $1/6$ (namely, $6\mathbb{Z}$) on which $\ell(b) \sim b^2$ as $b$ tends to infinity within that set. Any finite union of such sets $N_i\mathbb{Z}$, where $N_i$ has two distinct prime factors $p_i$ and $q_i$, will also have this property, and one may show with relative ease that such a union may be arranged to have natural density arbitrarily close to $1$. 

However, by quantifying estimates made in the previous lemma, we can do slightly better, and show the existence of a set with the desired property that has density $1$, thereby proving Theorem \ref{thm: ell nearly always maximal epsilon free}.

\begin{proof}[Proof of Theorem \ref{thm: ell nearly always maximal epsilon free}]
Let $f(N)$ be a function that satisfies $f(N)  \to \infty$ as $N \to \infty$ (to be further specified later). For integers $j \geqslant 0$, let $D_j$ denote the set of $n \in (2^{j-1},2^j]$ such that $n$ has at least two distinct prime factors $p,q \leqslant f(2^{j-1})$.
Let \[D:= \bigcup_{j\geqslant 0} D_j.\] 
The set $D$ is our candidate set for use in Theorem \ref{thm: ell nearly always maximal epsilon free}.

\begin{lemma} If $f$ grows slowly enough, the set $D$ has natural density $1$.
\end{lemma}

\begin{proof} We begin by fixing $j \geqslant 0$ and bounding the size of $D_j$ from below.  For convenience, we write $N$ for $2^{j-1}$.

To produce this lower bound, we find an upper bound for $(N,2N]\setminus D_j$. Indeed, by a standard application of a small sieve (e.g. the Selberg sieve, in particular Theorem 9.3.10 of \cite{Mu08}), one may show that the number of $n \in (N,2N]$ without any prime factor $p $ less than $f(N)$ is
\[O\bigg( N \prod_{p < f(N)} \left(1-\frac{1}{p} \right)\bigg),\]
provided $f(N)$ grows slowly enough.  By Mertens' Third Theorem, this quantity is $O(N/\log f(N))$.

By using a union bound and the sieve above, we bound the number of $n \in (N,2N]$ with exactly one prime factor $p < f(N)$ by
\[O\Bigg( \sum_{p < f(N)} \frac{N}{p}  \prod_{\substack{q<f(N) \\q \neq p}} \left(1-\frac{1}{q}\right)\!\Bigg).\]
This quantity is $O(N \log\log f(N)/\log f(N))$ (again by Mertens' theorems), and we conclude by exclusion that
\[\lvert D_j \rvert = N\left(1 - O \left(\frac{\log \log f(N)}{\log f(N)}\right)\right).\]

This already establishes that $D$ has full upper Banach density.  To show that $D$ has \emph{natural} density $1$, we fix $\varepsilon > 0$ and note that, since $f(N) \to \infty$ as $j \to \infty$, there exists $j_0(\varepsilon)$ such that $\lvert D_j \rvert \geqslant 2^{j-1}(1-\varepsilon)$ for all $j \geqslant j_0(\varepsilon)$.  In particular,
\begin{align*}
\sum_{\substack{n \leqslant X \\ n \in D}} 1 &\geqslant \sum_{j_0(\varepsilon) \leqslant j \leqslant \lceil \log_2 X \rceil} \vert D_j \vert - \sum_{X<n \leqslant 2^{\lceil \log_2 X \rceil}} 1. \\
&\geqslant (1-\varepsilon) \left(2^{\lceil \log_2 X \rceil} - 2^{j_0(\varepsilon)-1}\right) + X - 2^{\lceil \log_2 X \rceil}.
\end{align*}
Simplifying, we see that
\[ \liminf_{X \to \infty} \frac{\vert D \cap [1,X]\vert}{X} \geqslant \liminf_{X \to \infty} \frac{X - \varepsilon 2^{\lceil \log_2 X \rceil} - 2^{j_0(\varepsilon)}}{X} \geqslant 1 -2\varepsilon,\]
which implies that $D$ has natural density $1$, since $\varepsilon$ was arbitrary.
\end{proof}

Secondly, we prove that $\beta(n)$ is asymptotically large within $D$.

\begin{lemma} \label{second_lemma}
If $f$ grows slowly enough, then $\beta(n) \sim n$ as $n \to \infty$ within $D$.
\end{lemma}

\begin{proof} Our proof presents a more quantitative adaptation of the argument used in Lemma~\ref{thm:beta_asymp_along_AP}.  Let $\varepsilon > 0$, and fix $n \in D$.  By the definition of $D$, there exist distinct primes $p,q < f(n)$ for which $p,q \mid n$.  We will show that, provided $n$ is large enough in terms of $\varepsilon$, there exist non-negative integers $a$ and $b$ for which
\[e^{\varepsilon} \geqslant \frac{n}{p^a q^b} > 1.\]
Since $p^a q^b \leqslant \beta(n) < n$, and $\varepsilon$ is arbitrary, this will complete the proof.

Taking logarithms, it suffices to find non-negative integers $a$ and $b$ for which
\[\frac{\varepsilon}{\log q} \geqslant \log_q n - a \log_q p - b > 0.\]
Setting $L = \lfloor \log_p n \rfloor$, it will be enough to prove that the sequence of fractional parts $\{ \{a \log_q p\} : a \in [1,L]\}$ contains an element in every interval modulo $1$ of length $\varepsilon/\log q$.  Since $p,q \leqslant f(n)$, we reduce our theorem to the following claim:

\begin{claim} Let $L'=\lfloor \log n/\log f(n) \rfloor$.  Then $S=\{ \{a \log_q p\} : a \in [1,L']\}$ contains an element in every interval modulo $1$ of length $\varepsilon/\log f(n)$, provided $f(n)$ grows slowly enough.
\end{claim}

The proof of this claim follows from the Erd\H{o}s-Tur\'an inequality (Corollary 1.1 of \cite{Mo94}). Indeed, for any interval $I$ modulo $1$ of length $\varepsilon/\log f(n)$, we have
\begin{align} \label{eq:erdos_turan_bound}
\left\vert \vert S\cap I\vert - \frac{\varepsilon L'}{\log f(n)} \right\vert \ll \frac{L'}{K+1} + \sum_{k \leqslant K} \frac{1}{k} \bigg\vert \sum_{a=1}^{L'} e^{2\pi i a k\log_q p} \bigg\vert
\end{align}
for any integer $K \geqslant 1$. It suffices to show that we may choose a $K$ such that the right-hand side in~\eqref{eq:erdos_turan_bound} is $o(L'/\log f(n))$ as $n \to \infty$.

Choosing $K = \lfloor \log^2 f(n)\rfloor $ ensures that $L'/(K+1) = o(L'/\log f(n))$.  As for the second term in~\eqref{eq:erdos_turan_bound}, bounding the sum over $a$ as a geometric series gives
\[\sum_{k \leqslant K} \frac{1}{k} \bigg\vert \sum_{a=1}^{L'} e^{2\pi i a k\log_q p} \bigg\vert \leqslant G(K, p,q)\]
for some function $G$ that is independent of $L^\prime$. We may assume without loss of generality that $G$ is increasing in each variable. Then $$G(K,p,q) \ll G( \log^2 f(n), f(n), f(n) ),$$ so it suffices to show that
\begin{equation}
\label{growth function}
G\left(\log^2 f(n), f(n), f(n)\right) = o\left(\frac{L'}{\log f(n)}\right).
\end{equation}
Recalling the definition of $L'$, this is equivalent to showing
\[ G\left(\log^2 f(n), f(n), f(n)\right) \cdot \log^2 f(n)= o\left(\log n \right).\]
Yet $G$ is simply some absolute function, so if $f$ grows slowly enough then (\ref{growth function}) will hold. (If one so wished, one could quantify this growth condition using Baker's result \cite{Ba68} on linear forms of logarithms of primes). This proves the claim, and hence the lemma. 
\end{proof}
Combining Lemma~\ref{second_lemma} with Theorem~\ref{thm:main theorem} yields Theorem \ref{thm: ell nearly always maximal epsilon free}.
\end{proof}
\vspace{5 mm}
\bibliographystyle{plain}
\bibliography{aprestricteddigits}

\end{document}